\documentclass[11pt,a4paper]{article}
\usepackage[utf8x]{inputenc}
\usepackage{ucs}
\usepackage{amsmath}
\usepackage{amsfonts}
\usepackage{amssymb}
\usepackage{amsthm}
\usepackage{graphicx}
\usepackage{color}
\usepackage{geometry}
\usepackage{tikz}
\newtheorem{lem}{Lemma}
\newtheorem{rem}{Remark}
\newtheorem{thm}{Theorem}

\newcommand{\E}{\mathbb{E}}
\title{On the real zeros of random trigonometric polynomials \\with dependent coefficients}
\author{J\"urgen Angst, Federico Dalmao, Guillaume Poly}

\begin{document}

\maketitle

\begin{abstract}We consider random trigonometric polynomials of the form 
\[
f_n(t):=\sum_{1\le k \le n} a_{k} \cos(kt)  + b_{k} \sin(kt),
\]
whose entries $(a_{k})_{k\ge 1}$ and $(b_{k})_{k\ge 1}$ are given by two independent stationary Gaussian processes with the same correlation function $\rho$. Under mild assumptions on the spectral function $\psi_\rho$ associated with $\rho$, we prove that the expectation of the number $N_n([0,2\pi])$ of real roots of $f_n$ in the interval $[0,2\pi]$ satisfies
\[
\lim_{n \to +\infty} \frac{\mathbb E\left [N_n([0,2\pi])\right]}{n} = \frac{2}{\sqrt{3}}.
\]
The latter result not only covers the well-known situation of independent coefficients but allow us to deal with long range correlations. In particular it englobes the case where the random coefficients are given by a fractional Brownian noise with any Hurst parameter.
\end{abstract}

%\setcounter{secnumdepth}{3} 
%\setcounter{tocdepth}{3}
%\tableofcontents

\section{Introduction}
The study of roots and level lines of random functions is a central topic in mathematics, at the crossroad between Algebra, Analysis and Probability theory, which has been extensively studied since the mid 20th century. Being at the very definition of algebraic varieties, the investigation of the geometry of nodal sets associated to random algebraic or analytic functions is naturally of primer importance. In the pioneering work \cite{little1}, the authors considered the expected number of real zeros of random univariate algebraic polynomials with uniform, Gaussian, and discrete entries. Since then, lots of developments were made to estimate the asymptotic behavior, as their degree goes to infinity, of the real/complex level sets associated to such polynomials, under various assumptions on the law of the random entries, see for example \cite{kac1, kac2, erdos, ibragimov1,farah1, kostlan} and the references therein.
%\par
%\bigskip
\newpage
Among the class of random functions, of particular interest are random trigonometric polynomials of the form $\sum_{k=1}^n a_k \cos (k t)$ or $\sum_{k=1}^n a_k \cos (k t)+b_k\sin (k t)$, where $(a_k)_{k \geq 1}$ and $(b_k)_{k \ge 1}$ are random coefficients, since the distribution of the zeros of such polynomials occurs in a wide range of problems in science and engineering. The asymptotics of the mean number of real zeros of random trigonometric polynomials with independent standard and centered Gaussian coefficients was first explicited by Dunnage in \cite{dunnage}, where it is shown that this number is asymptotically proportional to the degree $n$ of the considered polynomial. Since then, the level sets of random trigonometric polynomials have been intensively investigated in various directions. For example, still in the case of independent standard and centered Gaussian entries, the variance and the fluctuations around the mean were studied in the serie of papers \cite{faravar,wigman,azais,azais2016}. Beyond the purely Gaussian case but still considering independent standard and centered entries, the universality of the asymptotic local/global behavior of the number of real zeros was recently established in \cite{angst2015universality,AZ15,flasche2017,iksanov2016local}.
\par
\bigskip
To the best of our knowledge, the case of dependent Gaussian entries in trigonometric models has only been considered in \cite{samba,renganathan1984average} which focus on the two particular and somehow ``extreme'' cases of a constant correlation $\E(a_i a_j)=\rho \in ]0,1[$ and a geometric correlation $\E(a_i a_j)= \rho^{|i-j|}$. In both cases, it is nonetheless shown that the expected number of real roots lying in $[0,2\pi]$ obeys the same asymptotics
\begin{equation}\label{Meanbehavior}
\lim_{n \to +\infty} \frac{\mathbb E\left [N_n([0,2\pi])\right]}{n} = \frac{2}{\sqrt{3}}.
\end{equation}
The latter result naturally raises the question of ascertaining the sharp conditions on the correlation function $\rho$ of the random coefficients ensuring this universal asymptotic behavior. The main result of this article, i.e. Theorem 1 p. 9 below, provides a significant step in that direction by exhibiting mild conditions on the spectral density $\psi_\rho$ guaranteeing that (\ref{Meanbehavior}) indeed holds. Our proofs exploit the Kac--Rice formula and more specifically the interpretation of its underlying integrand in terms of convolutions with respect to suitable trigonometric kernels. In a spirit close to \cite{ibragimov1971expected}, that is to say by investigating sign changes of piecewise linear approximations of the underlying random functions, the analoguous question for Kac polynomials ($P_n(x)= \sum_{k\le n} a_k x^k$) has been tackled in \cite{glendinning1989growth}. Under some assumptions on the spectral density, yet not covering the case of the increments of a fractional Brownian motion, it is shown that
\[
\lim_{n \to +\infty} \frac{\mathbb E\left [N_n(\mathbb{R})\right]}{\log(n)} = \frac{2}{\pi}.
\] 
\par
\bigskip
The plan of the article is the following. In the next Section 2, we introduce the considered model of random trigonometric polynomials with dependent coefficients and we explicit the hypotheses made on their correlation function. In Section 3, we introduce a suitable renormalization allowing us to express the covariance of the process and its derivative as convolutions with positive kernels which approximate unity. In the last Section 4, we use the celebrated Kac--Rice formula to express the expected number of real zeros and deduce its asymptotics.

\section{The model and the hypotheses}
We consider here random trigonometric polynomials of the form
\[
f_n(t):=\sum_{1\le k \le n} a_{k} \cos(kt)  + b_{k} \sin(kt), \quad t \in \mathbb R,
\]
where $(a_{k})_{k\ge 1}$ and $(b_{k})_{k\ge 1}$ are two independent sequences of standard centered Gaussian variables with correlation function $\rho : \mathbb N \to \mathbb R$, namely
$\mathbb E[a_k a_{\ell} ]= \mathbb E[b_k b_{\ell}] =:\rho(|k-\ell|)$ and $\mathbb E[a_k b_{\ell} ] =0$, $ \forall k, \ell \in \mathbb N^*.$
We will suppose that the spectral function $\psi_{\rho}$ 
\[
\psi_{\rho}(x) := \sum_{k \in \mathbb Z} \rho(|k|) e^{i k x}, \quad x \in ]0, 2\pi[,
\]
{{} exists} and satisfies the following hypotheses: \par
\vspace{-0.5cm}
\begin{equation}\label{eq.hypo}
\psi_{\rho} \in \mathbb L^1([0, 2\pi],dx), \quad \psi_{\rho} \; \text{is continuous on} \; ]0, 2\pi[ \;\; \text{and} \;\; \gamma_{\rho}:=\inf_{t \in [0, 2\pi]} \psi_{\rho}(t) >0.
\end{equation}\par
\vspace{-0.2cm}
\noindent
By Riemann--Lebesgue Lemma, the above integrability condition ensures that the correlation coefficient $\rho(k)$ goes to zero as $k$ goes to infinity, but no condition is required on the speed of the decay, allowing us in particular to consider long-range correlations.
For instance, the assumptions \eqref{eq.hypo} are satisfied by the fractional Gaussian noise with Hurst index $1/2<H<1$. Indeed, in this case, the correlation function $\rho_H$ is defined as \par
\vspace{-0.1cm}
\[
\rho_H({{}k}) := \frac{1}{2} \left( |1+{{}k}|^{2H} +|1-{{}k}|^{2H} -2|{{}k}|^{2H}   \right).
\]\par
\vspace{-0.1cm}
\noindent
Setting $c_H:=\sin(\pi H) \Gamma(2H+1)$ and as shown for example in Proposition 2.1 and Corollary 2.1 of \cite{beran1994}, , the associated spectral function is then \par
\vspace{-0.1cm}
\[
{{}\psi}_{\rho_H}(x) := 2 c_H ( 1-\cos(x) )  \sum_{j \in \mathbb Z}^{+\infty}  \frac{1}{(2\pi j +x)^{2H+1}}.
\]\par
\vspace{-0.1cm}
\noindent
This function admits a pole at the origin in accordance with the long-range correlation feature, with ${{} \psi}_{\rho_H}(x) \sim c_H |x|^{1-2H}$, and it admits a global positive minimum at $\pi$, as illustrated in Figure 1 below, for different values of the Hurst parameter $H$.\par
\vspace{-0.1cm}
\begin{figure}[ht]
\begin{center}
\includegraphics[scale=0.45]{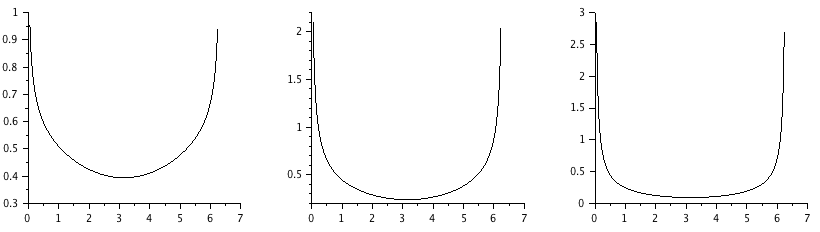}
\end{center}\par
\vspace{-0.5cm}
\caption{The graph of the function ${{}\psi}_{\rho_H}$ on $]0, 2\pi[$ for $H=0.6$, $H=0.75$, $H=0.9$.}
\end{figure}

\newpage
\section{Normalization and convergence}
Let us now introduce the following normalized version $F_n(t)$ of the process $f_n(t)$
\[
F_n(t):=\frac{1}{\sqrt{n}}\,  f_n(t) = \frac{1}{\sqrt{n}}  \left( \sum_{1\le k \le n} a_{k} \cos(kt)  + b_{k} \sin(kt) \right),
\]
and its derivative 
\[
F_n'(t):= \frac{1}{\sqrt{n}}  \left( \sum_{1\le k \le n} -k a_{k} \sin(kt)  + k b_{k} \cos(kt) \right).
\]
Naturally, the zeros of $f_n$ and its normalized version $F_n$ coincide. 
Moreover, the variance of $F_n(t)$ and $F_n'(t)$ have particularly nice expressions in terms of convolutions of the spectral function with non-negative kernels.

\begin{lem}\label{lem.var}
There exist two regular $2\pi-$periodic functions $K_n$ and $L_n$ such that for all $t \in [0,2\pi]$
\begin{equation}
\mathbb E [F_n(t)^2]  =  K_n \ast \psi_{\rho}(t), \qquad \mathbb E [F_n'(t)^2] =  \frac{(n+1)(2n+1)}{6} L_n \ast \psi_{\rho}(t).
\end{equation}
The two functions $K_n$ and $L_n$ are non-negative kernels with  $||K_n||_1 = ||L_n||_1 =1$ and 
\[
\forall \varepsilon>0, \;\; \lim_{n \to +\infty} \int_{\varepsilon}^{2\pi-\varepsilon} K_n(x) dx = \lim_{n \to +\infty} \int_{\varepsilon}^{2\pi-\varepsilon} L_n(x) dx =0.
\]
In particular, under the hypotheses of Section 2, both functions $K_n \ast \psi_{\rho}$ and $L_n \ast \psi_{\rho}$ converge uniformly to $\psi_{\rho}$ on any compact subset of $]0,2\pi[$.
\end{lem}

\begin{proof}
For fixed $t$ and $n$, we have
\[
\begin{array}{ll}
\mathbb E[F_n(t)^2] & \displaystyle{=\frac{1}{n} \sum_{k,l=1}^n \rho(k-l) \cos((k-l)t) = 1 + \frac{2}{n} \sum_{l>k}  \rho(k-l) \cos((k-l)t) }\\
 & \displaystyle{=  1 + \frac{2}{n} \sum_{k=1}^{n-1} \sum_{l=k+1}^n  \rho(k-l) \cos((k-l)t)=  1 +  \frac{2}{n}\sum_{k=1}^{n-1} \sum_{r=1}^{n-k}  \rho(r) \cos(r t) } \\
 & \displaystyle{=  1 +  \frac{2}{n}\sum_{r=1}^{n} (n-r) \rho(r) \cos(r t) =  \sum_{r=-n}^{n} \left(1-\frac{|r|}{n}\right) \rho(r) e^{irt} = K_n \ast \psi_{\rho}(t),} 
\end{array}
\]
where $K_n$ is the celebrated Fej\'er kernel, namely 
\[
K_n(x) := \sum_{r=-n}^{n} \left(1-\frac{|r|}{n}\right) e^{irx} = \frac{1}{n} \left(\frac{\sin(n x/2)}{\sin(x/2)}\right)^2.
\]
It is well known that $K_n$ has unit $\mathbb L^1-$norm and is an approximation of unity. By hypothesis, the spectral function $\psi_{\rho}$ is continuous on $]0, 2\pi[$, in particular it is uniformly continuous on any compact subset $K$ of $]0, 2\pi[$, so that the convolution $K_n \ast \psi_{\rho}$ uniformly converges to $\psi_{\rho}$ on $K$ as $n$ goes to infinity. 
In the same way, we have 
\[\label{e:var}
\begin{array}{ll}
\mathbb E[F_n'(t)^2] & \displaystyle{=\frac{1}{n} \sum_{k,l=1}^n k l \rho(k-l) \cos((k-l)t) = \frac{1}{n} \sum_{k=1}^n k^2 + \frac{2}{n} \sum_{l>k}  k l \rho(k-l) \cos((k-l)t) }\\
 & \displaystyle{=   \frac{1}{n} \sum_{k=1}^n k^2 + \frac{2}{n}\sum_{k=1}^{n-1} \sum_{r=1}^{n-k}  k(r+k)\rho(r) \cos(r t) } \\
  & \displaystyle{=   \frac{1}{n} \sum_{k=1}^n k^2 + \frac{2}{n}\sum_{r=1}^{n-1} \rho(r) \cos(r t) \left( \sum_{k=1}^{n-r} k(r+k) \right) } \\
    & \displaystyle{=   \sum_{r=-(n-1)}^{n-1} \rho(r) e^{i rt}  \left(  \frac{1}{n} \sum_{k=1}^{n-|r|} k(|r|+k) \right).} 
\end{array}
\]
With the convention that $\sum_{\emptyset} = 0$, and setting  $\alpha_n:=6/((n+1)(2n+1))$ for $n \geq 1$, we have thus 
\[
\mathbb E[F_n'(t)^2] =  \frac{1}{\alpha_n}  L_n \ast \psi_{\rho}(t),
\]
where
\[
L_n(x):= \alpha_n \sum_{r=-n}^{n}    \left(  \frac{1}{{n}} \sum_{k=1}^{n-|r|} k(|r|+k) \right) e^{i r x}.
 \]
Alternatively, the function $L_n(x)$ can be written as
\[
\begin{array}{ll}
L_n(x)& =\displaystyle{\alpha_n \sum^{n}_{k,l=1}kl\cos((k-l)x)
=\alpha_n \left( \left[\sum^{n}_{k=0}k\cos(kx)\right]^2 +\left[\sum^{n}_{k=0} k\sin(kx)\right]^2 \right)} \\
\\
& = \displaystyle{\alpha_n \left| \sum^{n}_{k=0} k e^{i kx}\right|^2 = \alpha_n \left| \frac{(n+1)e^{i (n+1)x}}{1-e^{ix}} - \frac{i e^{ix} \left(1-e^{i(n+1)x}\right)}{\left( 1-e^{ix} \right)^2}\right|^2}.
\end{array}
\]
It is therefore non-negative and satisfies the folllowing inequality
\[
 L_n(x)  \leq \alpha_n \left( (n+1) |1-e^{ix}|^{-1} + 2|1-e^{ix}|^{-2}\right), \quad \forall x \in ]0, 2\pi[.
\]
For all small $\varepsilon>0$, we have thus 
\[
\int_{\varepsilon}^{2\pi-\varepsilon} L_n(x) dx \leq\alpha_n \left(  (n+1) \sin(\varepsilon/2)^{-1} + 2\sin(\varepsilon/2)^{-2}\right) = O(1/n).
\]
Finally, we have
\[
||L_n||_1=\frac{1}{2\pi} \int_{-\pi}^{\pi} L_n(x)dx = \alpha_n \times \frac{1}{n} \sum_{k=1}^n k^2 =1.
\]
The next figure illustrates the behavior of the kernel $L_n$ for different values of $n$.
\begin{figure}[ht]
\begin{center}
\includegraphics[scale=0.4]{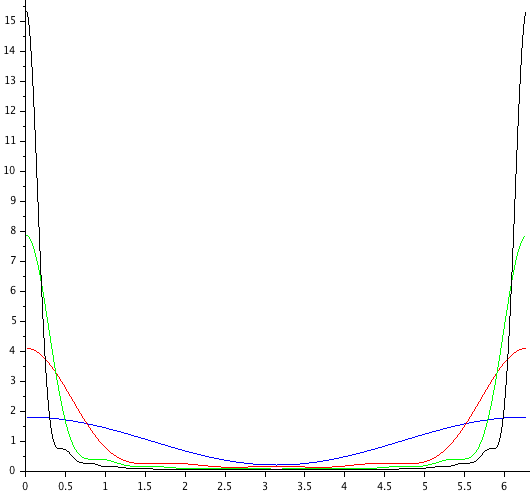}
\end{center}
\caption{The graph of the function $L_n$ on $]0, 2\pi[$ for $n=2,5,10,20$.}
\end{figure}
\end{proof} 

\begin{rem}\label{r:n-d}
The positivity of the Fej\'er kernel $K_n$ combined with the lower bound hypothesis on the spectral function ensures that the variance of $F_n(t)$ is bounded below, namely
\begin{equation}\label{eq.lbound}
\mathbb E[F_n(t)^2]  = K_n \ast \psi_{\rho}(t) \geq \gamma_{\rho} >0, \quad \forall t\in [0, 2\pi].
\end{equation}
\end{rem}
\par
\bigskip
Let us now describe the behavior as $n$ goes to infinity of the covariance between the process $F_n(t)$ and its derivative $F_n'(t)$. 
\begin{lem}\label{lem.covar}
The covariance between $F_n(t)$ and $F_n'(t)$ is given by 
\[
\mathbb E [F_n(t) F_n'(t)]  = \frac{1}{2} \, K_n' \ast \psi_{\rho}(t), 
\]
and under the hypotheses of Section 2, for any compact subset $K$ of $]0, 2\pi[$, it satisfies 
\[
\lim_{n \to +\infty} \frac{\sup_{t \in K} \left| \mathbb E [F_n(t) F_n'(t)] \right| }{n} = 0.
\]
\end{lem}

\begin{proof}The expression of the covariance as a convolution is simply obtained by differentiating the one of $\mathbb E[F_n(t)^2]$. The derivative of the Fej\'er kernel is given by the explicit formula
\[
K_n'(x) = \frac{\sin(nx/2)\cos(nx/2)}{\sin^2(x/2)}-\frac{\cos(x/2)\sin^2(nx/2)}{n \sin^3(x/2)},
\]
from which it is clear that for any small $\eta>0$, there exists a finite constant $R_{\eta}$ such that  
\begin{equation}\label{eq.primeup} 
\sup_{x \in [\eta, 2\pi-\eta]}|K_n'(x)| \leq R_{\eta}.
\end{equation}
Moreover, using Bernstein inequality, see e.g. Theorem 3.16 p.11 of \cite{zygmund2003}, we have 
\begin{equation}\label{eq.bernstein} 
||K_n'||_1 \leq n ||K_n||_1 = n.
\end{equation}
 Since the integral of $x \mapsto K_n'(x)$ over a period vanishes, for all $t \in K$, we can write 
\begin{equation}\label{eq.prime} 
\frac{1}{n} K_n' \ast \psi_{\rho}(t) = \frac{1}{2\pi} \int_{0}^{2\pi} \frac{K_n'(x)}{n} \left[ \psi_{\rho}(t-x) - \psi(t) \right] dx.
\end{equation}
For $\alpha>0$, let us denote by $K_{\alpha}$ the compact $\alpha-$neighboorhood of $K$ and let us fix $\alpha$ small enough so that $K_{\alpha} \subset ]0, 2\pi[$. The function $\psi_{\rho}$ is uniformly continuous on $K_{\alpha}$ and for all $\varepsilon>0$, there exists $0<\eta \leq \alpha$ such that $| \psi_{\rho}(t-x) - \psi(t) | \leq \varepsilon$ as soon as $|x|<\eta$. We can then decompose the right hand side of Equation \eqref{eq.prime} as the sum
\[
\frac{1}{n} K_n' \ast \psi_{\rho}(t) = A_n(t) + B_n(t), \]
where 
\[
A_n(t) := \frac{1}{2\pi} \int_{|x|<\eta} \frac{K_n'(x)}{n} \left[ \psi_{\rho}(t-x) - \psi(t) \right] dx
\]
and
\[
B_n(t):=\frac{1}{2\pi} \int_{\eta}^{2\pi-\eta} \frac{K_n'(x)}{n} \left[ \psi_{\rho}(t-x) - \psi(t) \right] dx.
\]
On the one hand, using the upper bound \eqref{eq.bernstein}, we have $|A_n(t)| \leq \varepsilon$ uniformly in $t \in K$. On the other hand, using this time the upper bound \eqref{eq.primeup}, we get that uniformly in $t$
\[
|B_n(t)| \leq \frac{R_{\eta}}{n} \times \left( ||\psi_{\rho}||_1 + \sup_{t \in K} |\psi_{\rho}(t)| \right),
\]
hence the result.
\end{proof}

\section{Asymptotics of the expected number of real zeros}
Thanks to the estimates for the variance and covariance established in the last section, we can now explicit the asymptotic behavior of the expected number of real roots of our random trigometric polynomial. We first consider the real zeros at a positive distance from the origin. Here $N_n(K)$ denotes the random number of real zeros of $F_n$ in a set $K$ whose volume i.e. Lebesgue measure is denoted by $\mathrm{vol}(K)$.
\begin{lem}\label{lem.bulk}
Let $K$ be a compact subset of $]0, 2\pi[$, under the hypotheses of Section 2, as $n$ goes to infinity, we have
\[
\mathbb E \left[ N_n(K) \right] = \frac{n}{\sqrt{3}\pi} \left( \mathrm{vol}(K) +o(1)\right).
\]
\end{lem}
\begin{proof}
The process $(F_n(t))_{t \geq 0}$ is a centered Gaussian process with $C^1-$paths. Besides, Remark 
\ref{r:n-d} above implies that for each $t\in[0,2\pi]$ the distribution of $F_n(t)$ is non-degenerated. 
Hence, we can use the celebrated Kac--Rice formula as in Theorem 3.2 in \cite{azaisW}  to compute the expectation of $N_n$. So let $K$ be a compact subset of $]0, 2\pi[$, 
the expected number of real zeros of $F_n$ in $K$ is then given by 
\[
\mathbb E \left[ N_n(K) \right] =\frac{1}{\pi} \int_{K} \sqrt{I_n(t)}dt,
\] 
where 
\[
I_n(t):=\frac{\mathbb E[F_n(t)^2]\mathbb E[F_n'(t)^2]-\mathbb E[F_n(t)F_n'(t)]^{2} }{\mathbb E[F_n(t)^2]^2 }.
\]
From Lemma 1, we have
\begin{equation}\label{eq.In}
I_n(t)= \frac{(n+1)(2n+1)}{6}\frac{L_n \ast \psi_{\rho}(t)}{K_n \ast \psi_{\rho}(t) } - \frac{\mathbb E[F_n(t)F_n'(t)]^2 }{(K_n \ast \psi_{\rho}(t))^2},
\end{equation}
and combining Lemma \ref{lem.var} and \ref{lem.covar}, we get that uniformly on $K$, as $n$ goes to infinity
\[
I_n(t)= \frac{(n+1)(2n+1)}{6}( 1 +o(1)), 
\] 
so that we have indeed
\[ 
\mathbb E \left[ N_n(K) \right]=n \, \frac{\text{vol}(K) }{\sqrt{3}\pi} \left( 1 +o(1)\right).
\]
\end{proof}
The expected number of real zeros in the neighborhood of the origin is handled thanks to the following Lemma.
\begin{lem}\label{lem.bord}
Under the hypotheses of Section 2, there exists a finite constant $C$ such that, for $\varepsilon>0$ small enough and for all $n \geq 1$
\[
\frac{\mathbb E \left[ N_n([0,\varepsilon]) \right]}{n} \leq C \sqrt{\varepsilon}, \qquad \frac{\mathbb E \left[ N_n([2\pi-\varepsilon,2\pi]) \right]}{n} \leq C \sqrt{\varepsilon}.
\]
\end{lem}
\begin{proof}
Again, thanks to Kac--Rice formula, we have
\[
\mathbb E \left[ N_n([0,\varepsilon]) \right] = \frac{1}{\pi} \int_{0}^{\varepsilon} \sqrt{I_n(t)}dt.
\]
Using Cauchy--Schwarz inequality, we have then 
\begin{equation}\label{eq.CS}
\mathbb E \left[ N_n([0,\varepsilon]) \right]^2 \leq  \frac{1}{\pi^2} \left(  \varepsilon \int_{0}^{\varepsilon} |I_n(t)|dt\right).
\end{equation}
Starting from the expression \eqref{eq.In} of $I_n(t)$, using the lower bound \eqref{eq.lbound} on the variance, we have for all $t \in [0, \varepsilon]$
\[
 |I_n(t)| \leq \frac{(n+1)(2n+1)}{6 \gamma_{\rho}} L_n \ast \psi_{\rho}(t) ,
 \]
 so that 
\[
\begin{array}{ll}
\displaystyle{ \int_{0}^{\varepsilon} |I_n(t)|dt } & \displaystyle{\leq \frac{(n+1)(2n+1)}{6 \gamma_{\rho}} \int_0^{2\pi} L_n \ast \psi_{\rho}(t)  dt} \\
 \\
 &  \leq  \displaystyle{\frac{2\pi(n+1)(2n+1)}{6 \gamma_{\rho}} || L_n \ast \psi_{\rho} ||_1 }\\
 \\
  &  \leq  \displaystyle{\frac{\pi(n+1)(2n+1)}{3 \gamma_{\rho}} || L_n ||_1 \times || \psi_{\rho} ||_1. }
 \end{array}
 \]
Remembering that $|| L_n ||_1=1$ and injecting this last estimate in Equation \eqref{eq.CS}, we get that for all $n\geq 1$
 \[
\frac{ \mathbb E \left[ N_n([0,\varepsilon]) \right]^2}{n^2} \leq  \frac{(n+1)(2n+1)}{3\pi  \gamma_{\rho}n^2} \varepsilon  \leq C^2 \varepsilon, \;\; \text{where} \;\; C:=\sqrt{\frac{2 || \psi_{\rho} ||_1}{\pi \gamma_{\rho}}}.
 \]
 The proof of the analogue estimate on $[2\pi-\varepsilon,2\pi]$ is similar.
\end{proof}

We can finally combine Lemma \ref{lem.bulk} and Lemma \ref{lem.bord} to deduce the asymptotic behavior of the expected number of real roots on the whole interval $[0, 2\pi]$.  
\begin{thm}Under the hypotheses of Section 2, as $n$ goes to infinity, we have
\[
\mathbb E \left[ N_n([0,2\pi]) \right] =  \frac{2n}{\sqrt{3}} \left( 1 +o(1)\right).
\]
\end{thm}
%\bibliographystyle{alpha}
%\bibliography{ref}

\newcommand{\etalchar}[1]{$^{#1}$}

\end{document}